\documentclass[10pt, final]{amsart}
\usepackage[english]{babel}
\usepackage{amssymb}
\usepackage{graphicx}
\usepackage{amscd}

\theoremstyle{plain}

\newtheorem{corollary}{Corollary}

\newtheorem{definition}{Definition}

\newtheorem{proposition}{Proposition}

\newtheorem{theorem}{Theorem}

\def\Ext{\operatorname{Ext}}

\newcommand{\iny}{\hookrightarrow} 

\newcommand{\N}{\mathbb{N}}

\def\Ext{\operatorname{Ext}}
\DeclareMathOperator{\fin}{FIN}
\DeclareMathOperator{\cofin}{COFIN}

\begin{document}

\title{Local complementation and  the extension of bilinear mappings}

\author{J.M.F. Castillo, A. Defant, R. Garc{\'{\i}}a, D. P\'erez-Garc{\'{\i}}a, J. Su\'arez}

\address{Departamento de Matem\'aticas, Facultad de Ciencias, Univ. de Extremadura,
Avda. de Elvas s/n, 06071 Badajoz (Spain).}
\email{castillo@unex.es} \email{rgarcia@unex.es}

\address{Instit\"{u}t fur Mathematik, Universit\"{a}t Oldenburg, Postfach 2503, 26111 Oldenburg (Germany).}
 \email{defant@mathematik.uni-oldenburg.de}
\address{Departamento de An\'alisis Matem\'atico, Facultad de CC Matem\'aticas, Univ. Complutense de Madrid, Pza. de Ciencias 3,  Ciudad Universitaria, 28040 Madrid (Spain).}
 \email{dperez@mat.ucm.es}


 \address{Escuela Polit\'ecnica, Universidad de Extremadura, Avenida de la Universidad s/n, 10071 C\'aceres, Spain.}
             \email{jesus@unex.es}

\thanks{The research of authors 1, 3 and 5 has been supported in part by MTM2010-20190-C02-01 and
Junta de Extremadura CR10113  ``IV Plan Regional I+D+i, Ayudas a
Grupos de Investigaci\'on''. The research of author 5 has been
supported in part by a grant inside the network ``Phenomena in
high dimensions" at the University of Oldenburg under the
supervision of professor A. Defant.}

\address{ }



\maketitle
\begin{abstract}
We study different aspects of the connections between local theory
of Banach spaces and the problem of the extension of bilinear
forms from subspaces of Banach spaces. Among other results, we
prove that if $X$ is not a Hilbert space then one may find a
subspace of $X$ for which there is no Aron-Berner extension. We
also obtain that the extension of bilinear forms from all the
subspaces of a given $X$ forces such $X$ to contain no uniform
copies of $\ell_p^n$ for $p\in[1,2)$. In particular, $X$ must have
type $2-\varepsilon$ for every $\varepsilon>0$. Also, we show that
the bilinear version of the Lindenstrauss-Pe\l czy\'{n}ski and
Johnson-Zippin theorems fail. We will then consider the notion of
locally $\alpha$-complemented subspace for a reasonable tensor
norm $\alpha$, and study the connections between $\alpha$-local
complementation and the extendability of $\alpha^*$ -integral
operators.
\end{abstract}

\section{Introduction}

The importance of the Hahn-Banach theorem in the linear theory of
Banach spaces has encouraged the study of possible non-linear
versions. Nevertheless, it is not difficult to see that there is
no hope for a general Hahn-Banach result in this non-linear
situation and the case of the scalar product in a Hilbert space is
clear: this bilinear form can be extended from $\ell_2$ to a
bigger superspace $X$ if and only if $\ell_2$ is complemented in
$X$. The same happens to every Banach space isomorphic to its dual
\cite{cagajiann}. There are however several interesting partial
results. Curiously enough, there seem to be only two types of
positive results: crude extension results and results involving
linearity of the extension process. The two main sections of this
paper, Sections 3 and 4, will study these two aspects of the
extension problem for bilinear forms separately, as we describe
now.

The general context for the problem is to take  $Y$ a subspace of
$X$ and consider the restriction operator $R:
\mathcal{B}(X)\rightarrow \mathcal{B}(Y)$ between the spaces of
bilinear (or, more generally,  multilinear) continous forms.\\

\begin{enumerate}

\item When is $R$ surjective?  (i.e., when does there exist an
extension, or else, a \emph{Hahn-Banach type} theorem  for
bilinear forms?). The problem is considered in Section
\ref{sec:Hahn}. Previous work can be found in \cite{cagajipro,
cagajiann, ca, manolo, hans, kiry, pe1}. Very recently, this
question has been related to Bell inequalities in
Quantum Information Theory \cite{bell}.  Our results in this line nicely complement those in \cite{fu}.\\
\\

\item When does, moreover, $R$ admit a linear and continuous
section? (i.,e when does there exist  an \textit{Aron-Berner type}
extension theorem for bilinear  forms?). This aspect of the
problem is considered in Section \ref{sec:morphism} and connected
with the complementation of the base spaces. Previous work in this
direction can be found in \cite{arbe, cagavi, cagajipro, ca, di,
gagamamu, za}.

\end{enumerate}

The paper is thus organized as follows: after this introduction
and a ``preliminaries" section, we treat in section 3 the
possibility of extension theorems for bilinear forms. A
combination of homological techniques and local theory of Banach
spaces shows that most classical extension theorems for linear
operators fail for bilinear forms: indeed, no Lindenstrauss-Pe\l
czy\'nski \cite{lipe} or Johnson-Zippin \cite{johnzipp} theorems
remain valid in the bilinear setting. Contrarily to what occurs
with  Maurey's extension theorem \cite{dijato}: we obtain an
analogue for  bilinear forms and show that if all bilinear forms
on subspaces of $X$ can be extended to $X$ then $X$ must have type
$2-\varepsilon$ for every $\varepsilon>0$ and weak type $2$ (see Theorem \ref{maurey}).\\

In section 4 we study the possibility of having linear continuous
extension for  bilinear forms, also called Aron-Berner extension.
It is well known \cite{cagajipro,di,za} that the existence of Aron-Berner extension
for bilinear forms implies local complementation, which
immediately implies, for instance, that the only subspaces of an
$\mathcal L_\infty$ (resp. $\mathcal L_1$) space admitting
Aron-Berner extension are those of the same type. Since the only
space all whose subspaces are locally complemented is the Hilbert
space, it follows that the only space for which there exists
Aron-Berner extension for all its subspaces is the Hilbert space (see Theorem \ref{thm:Hilbert}).
We will extend the result for  a finitely generated tensor norm (see Theorem \ref{alphalc}). \\

In section 5, we study the local character of multi-linear extension with respect
to a finitely generated tensor norm $\alpha$. We introduce the notion of local $\alpha$-extension (see definition \ref{alphaext})
and characterize when an $\alpha$-polynomial (resp. multilinear) may be extended to an $\alpha$-polynomial (resp. multilinear) (see Theorems \ref{rootns} and \ref{rootns1}).
Finally, we extend the notion of locally complemented subspace to a locally $\alpha$-complemented subspace, and show that it leads
to the extendability of $\alpha^*$ -integral operators (see Theorem \ref{rootns2}).

\section{Notation and Preliminares}

 In \cite{deflo} the interested reader can
find the most important tensor norms on tensor products of Banach spaces (
$\varepsilon$, $\pi$, $\omega _{2}$, ...). Given a Banach space $X$, we
denote by $\fin(X)$ or $\cofin(X)$, respectively, the set of all isometric finite dimensional or finite codimensional closed subspaces of
$X$.\\

Given a tensor norm $\alpha$ on $E_1\otimes...\otimes E_m$, we say
that $\alpha$ is finitely generated if for $z\in
\otimes_{\alpha}^m E_i$ the norm can be computed as
$\|z\|_{\otimes_{\alpha}^m E_i}=\inf
\{\|z\|_{\otimes_{\alpha}^m M_i}\, : \,M_i\in \fin(E_i)\}.$\\

Let $\alpha$ be a finitely generated tensor norm. We denote by
$\mathcal{A}(E_1,...,E_m):=(\widetilde{\otimes}_{\alpha}^m E_i)^*$
the ideal of multi-linear maps
associated to $\alpha$, see \cite{FH}. We denote the norm in this
space by $\|\cdot\|_{\mathcal{A}}$. Under the natural
identification, we will also refer to $\mathcal{A}(E_1,...,E_m)$ as
the space of $\alpha^*$-integral $m$-forms on $E_1\times...\times
E_m$, this is, the space of all $m$-linear maps $\varphi$ on
$E_1\times...\times E_m$
such that $L_{\varphi}\in \mathcal{A}(E_1,...,E_m)$ endowed with the norm $\|\varphi\|_{\alpha^*}:=\|L_{\varphi}\|_{\mathcal{A}}$. We denote $\mathcal{A}_m(E)$ if $E=E_{i}$ for all $i$.
\\

A Banach space $X$ is said to be an $\mathcal{L}_p$ for $1\leq p
\leq \infty$ if there exist a constant $\lambda>0$ such that for
every finite dimensional subspace $E$ of $X$ there exists another
 finite dimensional subspace $F$ of $X$ with $F\supseteq E$ and $d(F,\ell_p^{\dim
F})\leq \lambda$.\\

A Banach space $X$ is said to have \textit{type} $p$ (resp.
\textit{cotype} $q$) for $1\leq p\leq 2$ (resp. $2\leq q\leq
\infty $) if there is a constant $C$ such that for all finite
subsets $x_{1},\ldots ,x_{n}$ of $X$
\begin{equation*}
\left\Vert \sum r _{i}x_{i}\right\Vert _{L_2[0,1]}\leq C\left( \sum
\Vert x_{i}\Vert ^{p}\right) ^{1/p} \left(\text{resp. }\left( \sum
\Vert x_{i}\Vert ^{p}\right) ^{1/p}\leq C\left\Vert \sum
r_{i}x_{i}\right\Vert _{L_2[0,1]}\right),
\end{equation*}
where $r_i$ denotes the $i$-th Rademacher function. We recall that
$\mathcal{L}_{p}$-spaces have type $\min\{p,2\}$ and cotype
$\max\{2,p\}$ and that a Banach space has both type $2$ and cotype
$2$ if and only if it is isomorphic to a Hilbert space, see e.g.
\cite{pi}.

A Banach space $X$ is said to have {\it weak type} 2
\cite[p.~172]{PiVol} if there is a constant $C$ and a $\delta\in
(0,1)$, so that whenever $E$ is a subspace of $X$ and an operator
$T:E\to\ell_2^n$, there is an orthogonal projection $P$ on
$\ell_2^n$ of rank $>\delta n$ and an operator $S:X\to\ell_2^n$
with
$$Sx=PTx\;\;{\rm for\;all}\;\;x\in E,\;\;\;{\rm
and}\;\;\;\|S\|\le C\|T\|.$$

It is clear that type $2$ implies weak type $2$, although the
converse fails.

\subsection{Locally complemented subspaces and short exact sequence}
A short exact sequence of Banach spaces and linear continuous
operators is a diagram $$
\begin{CD}
0@>>>Y@>i>>X@>q>>Z@>>>0
 \end{CD}$$in which the image of each arrow coincides with the kernel of the following one. The open mapping theorem ensures that, given a sort exact sequence as before, $Y$ is a subspace of $X$ ($i$ is an injection map) and $Z$ is the corresponding quotient $X/Y$  ($q$ is a quotient map). The dual sequence
$$\begin{CD}
0 @>>> Z^* @>q^*>>X^* @>i^*>> Y^*@>>> 0
\end{CD}
$$ is also exact also. An  exact sequence is said to \emph{split} if $Y$ is complemented in $X$; which
means that there is a linear continuous projection $p: X\to Y$.
The sequence is said to \emph{locally split} if the dual sequence
splits.

\begin{definition}[\cite{K}] \label{def:local-comp}
Let $Y$ be a subspace of $X$ through a map $i$. We say that $Y$ is
locally complemented in $X$ (through $i$) if there exists a
constant $\lambda>0$ such that for every $F\in \fin(X)$ there
exist a retract $r_{F}:F \to Y$, this is $r_{F}i=id$, with
$\|r_{F}\|\leq \lambda$. For the quantitative version we will say
locally $\lambda$-complemented. i.e.; the exact sequence $0\to Y\to X \to Z\to 0$ $\lambda$-locally splits.
\end{definition}

Kalton proved in \cite{K} that the following conditions are equivalent:
\begin{enumerate}
 \item[(1)] $Y$ is locally $\lambda$-complemented in $X$ through $i$
  \item[(2)] $Y^*$ is $\lambda$-complemented in $X^*$ through $i^*$
  \item[(3)] Every finite rank operator $T:Y \to F$ can be extended to $X$, say by
  $\widetilde{T}$ with $\|\widetilde{T}\|\leq \lambda \|T\|$.\\
\end{enumerate}

It is therefore clear that $Y$ is locally ($\lambda$) complemented
in $X$ if and only if the exact sequence $0\to Y\to X \to Z\to 0$
($\lambda$) locally splits; and this happens if and only if the
dual sequence $0\to Z^*\to X^* \to Y^*\to 0$  ($\lambda$)
splits.\\

We borrow from homological algebra the notation $\Ext(Z,Y) =0$ to
mean that all exact sequences $0\to Y\to X \to Z\to 0$ split. In
\cite{cagajipro} a homological approach was  applied to the
problem of extendibility of bilinear forms from $Y$ to $X$ to
show:\\

\begin{enumerate}
  \item[$\bullet$] If  $\Ext(Y,(X/Y)^*)=0$ and  $\Ext(X/Y,X^*)=0$  then all bilinear forms on $Y$ extend to $X$.\\
\item[$\bullet$]   If  $\Ext(X/Y,Y^*)=0$ and  $\Ext(X,(X/Y)^*)=0$  then all bilinear forms on $Y$ extend to $X$.\\
\end{enumerate}

This immediately implies that bilinear forms on $\mathcal L_1$
subspaces of $\mathcal L_1$ spaces extend to the bigger space (as
well as bilinear form on subspaces of $\mathcal L_\infty$ spaces
inducing $\mathcal L_\infty$-quotients) (\cite{cagasu,K}).
And all bilinear forms on $Y$ extend to $\mathcal L_1$ if only if $\Ext(Y,(\mathcal L_1/Y)^*)=0$, equivalently $\Ext(\mathcal L_1/Y,Y^*)=0$ (see \cite[Thm. 2.3 ]{cagajipro} ).

\section{Hahn-Banach type extension theorems for bilinear forms}\label{sec:Hahn}
The problem of the extension of --even scalar!-- bilinear forms is
still mostly open. There is only a small number of papers
\cite{cagavi,cagajipro,cagajiann,ca,ca4,manolo,kiry} in which a
few techniques are developed.

Let us briefly recall the existing basic extension theorems for
linear continuous operators: the Hahn-Banach theorem, the
Lindenstrauss-Pe\l czy\'nski theorem \cite{lipe}, the
Johnson-Zippin theorem \cite{johnzipp} and Maurey's theorem
\cite{dijato}. Are there analogues for bilinear forms? \\

\subsection{Bilinear Lindenstrauss-Pe\l czy\'nski theorem.} The Lindenstrauss - Pe\l czy\'n ski theorem \cite{lipe} establishes
that every $C(K)$-valued operator defined on a subspace of $c_0$
can be extended to the whole $c_0$. Thus, can bilinear forms
defined on subspaces of $c_0$ be extended to $c_0$? The answer is
a sounding no. Let $(F_n)_n$ denote a sequence of finite
dimensional Banach spaces uniformly isomorphic to their duals.

\begin{proposition} Let $X$ be a Banach space containing a non-locally complemented subspace isomorphic to $c_0( F_n)$.
There is a bilinear form on $c_0(F_n)$ that cannot be extended to
$X$.
\end{proposition}
\begin{proof} Observe that the hipothesis already implies  $\lim dim F_n =+\infty$ since $c_0$ is locally complemented in any superspace. Let, for each $n\in \N$, $\tau_n: F_n\to F_n^*$ be an isomorphism so that $\sup_n \|\tau_n\|\|\tau_n^{-1}\|= M<+\infty$. Since $c_0(\mathcal F)$ is not locally complemented in $X$, the dual sequence
$$\begin{CD}
0 @>>> c_0(F_n)^\perp @>>>X^* @>>> \ell_1(F_n^*)@>>> 0
\end{CD}
$$ does not split, where $c_0(F_n)^\perp=(X/c_0(F_n))^* $.  This exactly means that any lifting of the
canonical inclusion $j_n: F_n^* \to \ell_1(F_n^*)$ to $X^*$ must
have norm greater than or equal to $g(n)$ with  $\lim g(n)=+\infty$.
Let $\sigma\in \ell_1$ be such that $\lim \sigma_n g(n)=+\infty$.  We define the ``diagonal"
operator
\begin{equation*}
D_\sigma: c_{0}(F_n )\longrightarrow \ell _{1}(F_n^{\ast })
\end{equation*}
given by $D[(f_{n})_{n}]= (\sigma_n \tau_n f_{n})_n$.
Let $J: c_0(F_n^*) \to c_0(F_n)$ be the isomorphism $J(g_n) = (\tau_n^{-1} g
_n)$. The existence of a lifting $D': c_0(F_n)\to X^*$ for $D$ would imply the existence of a lifting
$D'J$ for $DJ$ with norm at most $\|D'\|M$. Thus, there are uniformly bounded liftings for the restrictions $DJ_{|F_n^*} = \sigma_n j_n$, in contradiction with the assumption $\lim \sigma_n g(n)=+\infty$.
\end{proof}

\begin{corollary} Let $X$ be a Banach space containing uniform copies of $\ell_2^n$ not uniformly complemented. There are bilinear forms on some subspace of $c_0(X)$ that cannot be extended to the whole $c_0(X)$.
\end{corollary}
\begin{proof} BY the hypothesis, $X$ contains copies $j_n: \ell_2^n\to X$, hence $c_0(X)$ contains $c_0(\ell_2^n)$ via the natural isomorphism $(j_n): c_0(\ell_2^n) \to c_0(X)$. If this $c_0(\ell_2^n)$ is locally complemented, the canonical projections $p_n: c_0(\ell_2^n) \to \ell_2^n$ could be uniformly extended to $c_0(X)$, yielding uniformly complemented copies of $\ell_2^n$, against the hypothesis.
\end{proof}

An obvious corollary is:

\begin{corollary}
There is a bilinear form on a subspace of $c_0$ that cannot be
extended to $c_{0}$.
\end{corollary}

This shows that no bilinear Lindenstrauss-Pe\l czy\'nski theorem
about the extension of $C(K)$-valued operators on subspaces of
$c_0$ is possible. \\

\noindent \textbf{Remarks.} There are aspects of this problem not
well understood. Observe that a subspace $Y$ of $c_0$ is locally
complemented if and only if it is an $\mathcal L_\infty$-space, in
which case \cite{Alspach} it must be isomorphic to $c_0$. Thus,
$c_0(F_n)$ is not locally complemented in $c_0$ unless it is
isomorphic to $c_0$; which implies that the spaces $F_n$ must be
uniformly complemented in $c_0$, hence in some $\ell_\infty^m$. It
is however a hard open problem, see \cite{moreplic}, to determine
if a sequence $F_n$ of finite dimensional Banach spaces uniformly
complemented in $\ell_\infty^m$ must be uniformly isomorphic to
$\ell_\infty^k$. On the other hand, it follows from
\cite{cagajipro} that if a subspace $Y$ of $c_{0}$ admits the
extension of bilinear forms to $c_0$ then it is a Hilbert-Schmidt
space, namely $\mathfrak L(Y, \ell_2) = \Pi_2(Y, \ell_2)$. For a
subspace of the form $c_0(F_n)$ this condition implies that there
is a uniform constant $C$ so that for every $n$ the
$\pi_2$-summing norm and the standard operator norm are
$C$-equivalent on $\mathfrak L(F_n, \ell_2^n)$. We do not know if
this condition implies that the $F_n$ are uniformly complemented
in $\ell_\infty^m$. All this together suggests:

\noindent \textbf{The $c_0$-conjecture.} A subspace of $c_0$ on
which every bilinear form can be extended to the whole $c_0$ is
isomorphic to $c_0$.\\

\subsection{Bilinear Johnson-Zippin theorem.} The Johnson-Zippin theorem \cite{johnzipp} establishes that $\mathcal{L}_{\infty
}$-valued operators defined on weak*- subspaces of $\ell
_{1}$ can be extended to the whole $\ell_1$. Again, there is no
bilinear Johnson-Zippin theorem. Indeed, one has

\begin{proposition} Let $(F_n)$ be a sequence of finite dimensional Banach spaces that is dense, in the Banach-Mazur distance, in the family of all finite dimensional Banach spaces. If $H$ is any subspace of $c_0$
such that $H^* =\ell_1(F_n)$ then there are bilinear forms on $H^\perp$  that cannot be extended to
 $\ell_1$.
\end{proposition}
\begin{proof}  Standard facts from homological algebra (see \cite{cagajiann}, or else \cite{temps}) yield that
for a subspace $A$ of $\ell_1$ the assertion ``all operators
$\mathfrak L(A,B)$ can be extended to $\ell_1$" can be
reformulated as $\Ext(\ell_1/A, B)=0$. Hence, given $H\subset c_0$
that all bilinear forms on $H^{\bot}$ --i.e., operators from
$\mathfrak L(H^\perp, H^{**})$-- can be extended to $\ell_1$ is
equivalent to $\Ext(H^*, (c_0/H)^{**}) = 0$, see \cite[Theorem
2.3]{cagajiann}. In the situation we are considering this means
$\Ext(\ell_1(F_n), (c_0/H)^{**}) = 0$. But this means that
$\Ext(F_n, (c_0/H)^{**}) = 0$ uniformly on $n$ --see
\cite{snark}-- from where it follows that $(c_0/H)^{**}$ is an
$\mathcal{L}_{\infty}$-space. But then also $c_0/H$ must be an
$\mathcal{L}_{\infty}$-space, which would make $\ell_1(F_n)$ an
$\mathcal L_1$-space, which is impossible.
\end{proof}
An immediate corollary is:
\begin{corollary} Every $\mathcal L_1$-space admits a bilinear form on a subspace
that cannot be extended to the whole space.
\end{corollary}

This is somewhat surprising since every bilinear form on $\mathcal L_1$-spaces can be extended to
$\mathcal L_1$-superspaces (see \cite{cagajiann}). The surprise passes when one is shown that the
case of $\ell_1$ is essentially the unique possible: it has been
shown in  \cite{manolo} that if a subspace $Y$, with unconditional
basis, of an $\mathcal{L}_1$-space is such that every $m$-linear
form $\varphi$ on $Y$ extends to one $\widetilde{\varphi}$ on the
whose space with $\|\widetilde{\varphi}\|\le C^m \|\varphi\|$,
then $Y=\ell_1$.\\

After this shock of negative results let us present a positive one. It is well known that Maurey's
extension theorem admits a bilinear version so we provide a sort of converse.

\begin{theorem}\label{maurey}$\;$
\begin{enumerate}
\item (Maurey) If $X$ has type $2$, then every bilinear form on a subspace
of $X$ can be extended to the whole $X$.

\item If every bilinear form on every subspace of $X$ can be
extended to $X$ then $X$ must have type $2-\varepsilon $ for every
$\varepsilon>0$. \end{enumerate}\end{theorem}

\begin{proof}
It is clear that the first claim goes back to Maurey's extension
Theorem \cite[Corollary 12.23]{dijato}. For the second, let us
assume that there exists $p<2$ such that $\ell_p$ is finitely
representable in $X$. Then $\ell _{p}^{n}$ is uniformly embedded in
$X$ via $j_{n}:\ell _{p}^{n}\hookrightarrow X$. By the construction
given in \cite[Theorem 3.1]{bedo} to obtain an un-complemented copy
of $\ell _{2}$ in $L_{p}$ we can consider $T_{n}:\ell
_{2}^{n}\rightarrow E_{n}\subset L_{p}$ to be uniformly bounded
embeddings with no uniformly
bounded projections. Moreover, for every $n$, there exist $%
(1+\varepsilon )$-embeddings $\Phi _{n}:E_{n}\rightarrow \ell
_{p}^{k(n)}\subset L_{p}$, for $k(n)$ sufficiently large. So we
have embeddings $i_{n}=\Phi _{n}\circ T_{n}$ $:\ell
_{2}^{n}\rightarrow \ell _{p}^{k(n)}$, that are uniformly bounded
and with no uniformly bounded projections:\\
$$
\begin{CD}
 \ell _{2}^{n} @>i_n>> \ell _{p}^{k(n)}  @>j_n>> X\\
  @VI_nVV  @V\widetilde{I}_{n}VV @V\widetilde{I}_{n}VV \\
\ell _{2}^{n} @<<< \ell_{p^{\ast }}^{k(n)} @<<< X^*\\
\end{CD}
$$
\\

\noindent where $I_{n}:\ell _{2}^{n}\rightarrow \ell _{2}^{n}$ is
the identity (Riesz-duality) associated to the bilinear scalar
product. Since we can extend bilinear forms from subspaces of $X$,
we can do it with a uniform constant (by the closed graph theorem).
And this contradicts the fact that the $i_n$'s do not have uniformly
bounded projections. Since $\ell_p$ is not finitely representable
for $p\in[1,2)$, by Maurey-Pisier Theorem \cite[Theorem 3.11]{pi},
$X$ has type $2-\varepsilon$ for every $\varepsilon>0$.
\end{proof}

The result is somehow optimal since the converse to (2) fails: the
Kalton-Peck space $Z_2$ \cite{kape} has type $2-\varepsilon$ for
every $\varepsilon>0$ and contains an uncomplemented copy of
$\ell_2$. Consequently, the canonical bilinear form on $\ell_2$
cannot be extended to $Z_2$.\\

Nevertheless, we can still improve it a little turning in a
different direction. Recall that a Banach space $X$ is said to en
enjoy the Maurey extension property (MEP) if for every subspace
$E\subset X$ all linear continuous operators $E\to \ell_2$ can be
extended to operators $X\to \ell_2$. It is not to hard to show
that MEP can be reformulated in finite dimensional terms: there is
a constant $C>0$ such that for all finite dimensional subspaces
$F\subset X$ all norm one operators $F\to \ell_2$ can be extended
to operators $X\to \ell_2$ with norm at most $C$. Type $2$ spaces
have MEP, and a theorem of Milman and Pisier
\cite[Th.~10]{MilmanPisier} establishes that $X$ must have weak
type $2$. One has:

\begin{proposition} If a Banach space $X$ is such that all bilinear forms on all subspaces can be extended to $X$ then
$X$ has the Maurey extension property, hence it must have weak
type 2.\end{proposition}
\begin{proof} The failure of MEP means the existence of a sequence $(F_n)$ of finite dimensional
subspaces of $X$ and norm one operators $\phi_n: F_n\to \ell_2$
such that any extension of $\phi_n$ to $X$ has norm at least
$g(n)$ with $\lim g(n)=+\infty$. Let $\sigma\in \ell_1$ such that
$\lim \sigma_n g(n)= + \infty$. Standard techniques ---see
\cite{moreplic}-- allow us to assume that $(F_n)$ forms a
finite-dimensional decomposition of its closed linear span $\Sigma
F_n$. This subspace of $X$ admits the bilinear form $\Phi: \Sigma
F_n \to (\Sigma F_n)^*$ defined as

$$\Phi((f_n))(g_n) = \sum \sigma_n <\phi_n f_n, \phi_n g_n>,$$

which cannot be extended to $X$.

\end{proof}

There are other situations in which it is possible to extend
bilinear forms: such is the case of  $\mathcal L_\infty$-spaces
such as  $A(\mathbb{D})$ -the disk algebra-, $H^{\infty }$ or
Pisier's spaces  ${P}$ verifying ${P}\widehat{\otimes
}_{\varepsilon }{P}={P}\widehat{\otimes }_{\pi }{P}$. Actually, it
was proved in \cite[Theorem 2.2]{cagajipro} that every bilinear
form on $Y$ extends to any super-space iff $Y\widehat{\otimes
}_{\omega_2}Y=Y\widehat{\otimes }_{\pi}Y$.\\

We conclude this section with a final remark. One additional
obstruction to the possibility of obtaining extension theorems for
bilinear forms is the following: bilinear forms are not, as a
rule, weakly sequentially continuous (wsc, in short); and the
scalar product in a Hilbert space is the perfect example.
Nevertheless, all $C(K)$-spaces, as all spaces having the
Dunford-Pettis property, have all bilinear forms wsc. Therefore
the only chance to have an extension theorem for bilinear forms
from a Banach space $X$ to some superspace $C(K)$ is that all
bilinear forms on $X$ must be wsc. This conditions is however not
enough, since all subspaces of $c_0$ have the Dunford-Pettis
property and, as we have already shown, many of them admit
bilinear forms that cannot be extended to $c_0$; the extension to
bigger $C(K)$-spaces is impossible too by \cite[Proposition
2.1]{cagajipro}. On the positive side, bilinear forms on
$A(\mathbb{D})$ or on $\mathcal L_\infty$-spaces can be extended
to any superspace \cite[Example 1]{cagajipro}.

\section{Linear continuous extension of bilinear forms}\label{sec:morphism}

The Aron-Berner \cite{arbe} type extensions of bilinear forms is
linear continuous extension and  was shown to be equivalent to
local complementation in \cite{cagajipro,di,za}.

\begin{proposition}\label{prop:castillo}
The following conditions are equivalent:

\begin{itemize}
\item[(1)]  $Y$ is locally complemented in $X$ through $j$.

\item[(2)]  $Y\widehat{\otimes }_{\pi }Y$ is locally complemented in $X%
\widehat{\otimes }_{\pi }X$ (through $j\otimes j$).

\item[(3)]  There exists an extension operator $\Phi_{2}
:\mathcal{B}(Y)\rightarrow \mathcal{B}(X)$. i.e.; $\mathcal{B}(Y)$ is complemented in $\mathcal{B}(X)$.
\item[(3)]  There exists an extension operator $\Phi_{n}
:\mathcal{L}(^nY)\rightarrow \mathcal{L}(^nX)$. i.e.; $\mathcal{L}(^nY)$ is complemented in $\mathcal{L}(^nX)$.
\end{itemize}
\end{proposition}
In particular, since $\mathcal{L}_{\infty}$-spaces are locally
complemented in every subspace, they always admit linear
continuous extension of bilinear forms to any superspace. A
partial positive answer to the $c_0$-conjecture can be presented.

\begin{proposition}
A subspace $Y$ of $c_0(\Gamma)$ admits  the extension of bilinear
forms to $c_0(\Gamma)$ in a linear and continuous form if and only
if it is isomorphic to some $c_0(I)$. The same is true for
n-linear maps.
\end{proposition}
\begin{proof}
By the proposition above, $Y$ must be locally complemented in
$c_0(\Gamma)$, hence it must be an $\mathcal{L}_{\infty}$ space,
and therefore \cite{gkl}, it must be isomorphic to some $c_0(I)$.
\end{proof}

Proposition \ref{prop:castillo} easily extends to the multilinear
setting.

\begin{theorem}\label{alphalc}
Let be $j:Y \hookrightarrow X$ and
$\alpha$ a finitely generated $m$-tensor norm. The following
conditions are equivalent:
\begin{enumerate}
  \item $Y$ is locally complemented in $X$ (through $j$).
  \item  $\widetilde{\otimes}_{\alpha}^m Y $ is locally complemented
  in $\widetilde{\otimes}_{\alpha}^m X $ (through $\otimes^m_{\alpha} j$)
  \item $\mathcal{A}_m(Y)$ is complemented in $\mathcal{A}_m(X)$
\end{enumerate}
\end{theorem}

\begin{proof}
$(1)\Rightarrow(2)$ Take $F\in \fin({\otimes}_{\alpha}^m
X)$ and notice that $F\subseteq \otimes^m F_i$ with $F_i\in
\fin(X)$ . Consider for each $i$ a retraction $r_{F_i}:F_i\to Y$
and thus over $F$ we have the retraction $\otimes^m r_{F_i}$
(restricted to F).

$(2)\Rightarrow(1)$  follows from the general fact: ``if $Y$ is
locally complemented in $X$ and $Y_0\subseteq Y$, $X_0\subseteq X$
are complemented subspaces such that $Y_0\subseteq X_0$ then $Y_0$
is locally complemented in $X_0$.''

$(2)\Leftrightarrow(3)$ is exactly the equivalence
$(1)\Leftrightarrow(2)$ in \cite[Theorem. 3.5]{K}.
\end{proof}

The proof of $(1)\Rightarrow(2)$ yields that when $Y$ is
$\lambda$-locally complemented in $X$, then every
$\alpha$-integral $m$-form $T$ on $Y$ can be extended to an
$\alpha$-integral $m$-form $\widetilde{T}$ on $X$ with
$\|\widetilde{T}\|\leq \lambda^{m+1}\|T\|$. This condition of
extension with control of the norms was recently used in
\cite{manolo} and is the right condition to deal with extensions
of holomorphic functions, see \cite{manolo}. On the other hand
taking an uncomplemented copy of $\ell_1$ inside $\ell_1$, see
 \cite{bourgain}, one can see that this extension with control of the norms is not equivalent to local
complementation.\\

There is a partial converse to Proposition \ref{prop:castillo},
which can be understood as the local version of the classical
Lindentrauss-Tzafriri theorem \cite{litz}. It essentially appears
in Fakhoury \cite[Thm. 3.7]{fa}, implicitly in \cite[Lemma
12.4.3]{alka}) and has been recently and independently obtained by
Fern\'andez Unzueta and Prieto \cite[Thm. 2.1]{fu}. For the sake
of completeness we include a clean draft for the proof.

\begin{theorem} \label{thm:Hilbert}
A Banach space $X$ is a Hilbert space if and only if every
subspace is locally complemented .
\end{theorem}
\begin{proof}
For the non-trivial implication let us take an infinite
dimensional space $X$ such that all of its subspaces are locally
complemented. Let us define $\lambda _{X}(Y)$ as the infimum of
the local complementation constants of $Y$ in $X$ and
\begin{equation*}
\lambda _{s}(X)=\sup \{\lambda _{X}(Y):Y\subset X\}\; .
\end{equation*}
Assume that $\lambda _{s}(X)=\infty $. First, is clear that for every subspace $Y$  of finite  codimension  $\lambda _{s}(Y)=\infty $.

Let $E_{1}$ be a finite dimensional subspace of $X$ with $\lambda
_{X}(E_{1})\geq 1$. Using a standard argument a local theory (see,
\cite[Theorem 1]{dadesi} or \cite[Lemma 12.4.3]{alka}) there is a
finite codimensional subspace $Y_{1}$  such that the natural
projection of $E_{1}\oplus Y_{1}$ to $E_{1}$ has norm $\leq 2$.


We now inductively construct a sequence of disjoint finite
dimensional subspaces $(E_{n})$ and finite codimensional subspaces
$(Y_{n})$ such that
 $$E_{n+1}\subset Y_{n} ; ~~
 Y_{n+1}\subset Y_{n}, ~~  \lambda _{X}(E_{n})\geq n$$
 and such
that the projection constant from $E_{1}\oplus \cdots \oplus
E_{n}\oplus Y_{n}$ to $E_{1}\oplus \cdots \oplus E_{n}$ has norm
$\leq 2$. Let $Y$ be the following subspace of $X$: $$Y =\{\sum
e_{n}: e_{n}\in E_{n}$$ and $\sum e_{n}$ converges in $X\}$. $Y$
is a subspace of $X$ and has the Schauder decomposition $(E_{n})$.
It is clear that the natural projection of $Y$ to $E_{n}$ has norm
at most $4$, which implies that $Y$ cannot be locally complemented
in $X$.

All this together means that we have shown that $\lambda
_{s}(X)<\infty $. In particular, this implies that there is a
uniform constant --$\lambda_s(X)$-- co that all finite dimensional
subspaces of $X$ are $\lambda_s(X)$-complemented in $X$. What
remains is to follow the final step in the proof of the
Lindenstrauss-Tzafriri classical theorem; see e.g. \cite[Thm
12.4.2]{alka}.\end{proof}

This characterization should be compared with the type of results
in \cite{fa}.

\begin{corollary}
There is no Banach space $X$ such that every infinite dimensional
subspace is an $\mathcal{L}_\infty$-space.
\end{corollary}

\begin{corollary} Given any Banach, non-Hilbert, space $X$ there always exist a weak*-closed uncomplemented subspace of $X^*$.
\end{corollary}

\section{Localization with respect to other tensor norms}

It is part of the common believing that the extension of
multi-linear forms is a {\it local} phenomenon. A different thing
is to precise in which sense or senses this is true. Here we
present a natural generalization of an idea presented in
\cite{cagajipro}. We need to introduce some definitions.

By a normed ideal of $n$-linear continuous operators between
Banach spaces (\cite{pie}) we mean a pair $(\mathcal{A},
\|\cdot\|_{\mathcal{A}})$ such that
\begin{enumerate}
    \item $\mathcal{A}(E_1,...,E_n;F)=\mathcal{A}\cap
    \mathcal{L}(E_1,...,E_n;F)$ is linear and
    $\|\cdot\|_{|\mathcal{A}(E_1,...,E_n;F)}$ is a norm.
    \item If $T_j\in \mathcal{L}(G_j;E_j),
    \varphi\in\mathcal{A}(E_1,...,E_n;F)$ and $S\in
    \mathcal{L}(F;G)$ then the composition $S\circ\varphi\circ
    (T_1,...,T_n)$ is in $\mathcal{A}$ and $\|S\circ\varphi\circ
    (T_1,...,T_n)\|_{\mathcal{A}}\leq \|S\|\|\varphi\|_{\mathcal{A}}\|T_1\|...\|T_n\|.$
    \item The map $\{\mathbb{K}^n\ni (x_1,...,x_n)\to x_1\cdot\cdot\cdot x_n \in
    \mathbb{K}\}$ belongs to $\mathcal{A}$ with corresponding
    $\mathcal{A}$-norm equal 1.
\end{enumerate}

An ideal $(\mathcal{A}, \|\cdot\|_{\mathcal{A}})$ is called
maximal if the condition

\begin{center}
$\|\varphi\|_{\mathcal{A}^{\max}}:=\sup \{ \|Q_L^F \circ
\varphi_{|M_1\times...\times M_n}\|_{\mathcal{A}}\;\;:\;\;M_j\in
\fin(E_j),L\in \cofin(F)  \}<\infty$
\end{center}
implies that $\varphi \in \mathcal{A}$ and
$\|\varphi\|_{\mathcal{A}}=\|\varphi\|_{\mathcal{A}^{\max}}$
holds (where $Q_L^F:F\to L$ is the projection). Given an ideal $(\mathcal{A}, \|\cdot\|_{\mathcal{A}})$,
its maximal envelope is defined as the linear set
$\mathcal{A}^{\max}:=\{ \varphi\in
\mathcal{M}(E_1,...E_n):\;\;\|\varphi\|_{\mathcal{A}^{\max}}<\infty
\}$ endowed with the norm $\|\cdot\|_{\mathcal{A}^{\max}}$.\

A subclass $\mathcal{Q}\subset \mathcal{P}^n$ of $n$-homogeneous
continuous scalar-valued polynomials on Banach spaces is called an
ideal if
\begin{enumerate}
    \item $\mathcal{Q}(E)=\mathcal{P}^n\cap
    \mathcal{Q}$ is a linear subspace of $\mathcal{P}^n(E)$ for all Banach spaces $E$.
    \item If $T\in \mathcal{L}(E;F),
    q\in\mathcal{Q}(F)$ then $q\circ T \in \mathcal{Q}$.
    \item The map $\{\mathbb{K}\ni z \to z^n \in
    \mathbb{K}\}$ belongs to $\mathcal{Q}$.\\
If $\|\cdot\|_{\mathcal{Q}}:\mathcal{Q}\to [0,\infty]$
    satisfies\

\item $\|\cdot\|_{\mathcal{Q}}|_{\mathcal{Q}(E)}$ is a norm for
all Banach spaces $E$.
    \item $\| q\circ T\|_{\mathcal{Q}}\leq
    \|T\|^n\|q\|_{\mathcal{Q}}$ in the situation of (2).
    \item $\|\{\mathbb{K}\ni z \to z^n \in
    \mathbb{K}\}\|_{\mathcal{Q}}=1$,\
\end{enumerate}
then $(\mathcal{Q},\|\cdot\|_{\mathcal{Q}})$ is called a normed
ideal of $n$-homogeneous polynomials. For
$(\mathcal{Q},\|\cdot\|_{\mathcal{Q}})$ and $q\in \mathcal{P}^n(E)$
define $\|q\|_{Q_{\max}}:=\sup\{\|q|_M\|_{\mathcal{Q}}\;\;|M\in
\fin(E) \}\in [0,\infty].$ $(\mathcal{Q},\|\cdot\|_{\mathcal{Q}})$
is called maximal if every $q\in \mathcal{P}^n(E)$ with
$\|q\|_{Q_{\max}}<\infty$ is in $\mathcal{Q}$ and
$\|q\|_{Q}=\|q\|_{Q_{\max}}$. Given an ideal $(\mathcal{Q},
\|\cdot\|_{\mathcal{Q}})$, we define its maximal envelope, as
before, like $\mathcal{Q}^{\max}:=\{ q\in
\mathcal{P}^n(E):\;\;\|q\|_{\mathcal{Q}_{\max}}<\infty \}$ endowed
with the norm $\|\cdot\|_{\mathcal{Q}_{\max}}$.
$(\mathcal{Q},\|\cdot\|_{\mathcal{Q}})$ is called ultrastable if for every ultrafilter $\mathfrak{U}$ the following holds:
For $q_{\iota}\in \mathcal{Q}(E_{\iota})$ with $\sup_{\iota\in
I}\|q_{\iota}\|_{\mathcal{Q}}<\infty$ one has
$\lim_{\mathfrak{U}}q_{\iota}\in
\mathcal{Q}((E_{\iota})_{\mathfrak{U}})$ and
$\|\lim_{\mathfrak{U}}q_{\iota} \|_{\mathcal{Q}}\leq sup
\|q_{\iota}\|_{\mathcal{Q}}$ (and hence $\leq
\lim_{\mathfrak{U}}\|q_{\iota}\|_{\mathcal{Q}}$).\\

\begin{definition}\label{alphaext}
Let $j_i:Y_i \to X_i$ be embeddings, $i=1,...,n$, and $\alpha$ a finitely
generated tensor norm. We say that the multi-linear mapping
$\varphi:\prod Y_i \to \mathbb{K}$ admits a \textbf{local
$\alpha$-extension} to $\prod X_i$ through $\prod j_i$ if there
exists a constant $\lambda$ such that $\forall E_i \in
\fin(Y_i),F_i\in \fin(X_i)$ we have that, for $i_{E}:\prod E_i
\to \prod Y_i$, $\varphi i_E:\prod E_i \to \mathbb{K}$ admits an
extension $\widetilde{\varphi} :\prod (E_i+F_i) \to \mathbb{K}$
through $(j_i)_{i=1}^m$ with $\|\widetilde{\varphi}\|_{\alpha^*}\leq
\lambda \|\varphi i_E\|_{\alpha^*}$. Equivalently, $L_{\varphi i_E}
\in(\widetilde{\otimes}_{\alpha}^m E_i)^*$ admits an extension
$L_{\widetilde{\varphi}}\in(\widetilde{\otimes}_{\alpha}^m
(E_i+F_i))^*$  with $\|L_{\widetilde{\varphi}}\|_{\mathcal{A}}\leq
\lambda \|L_{\varphi i_E}\|_{\mathcal{A}}$.
\end{definition}

For the quantitative version we shall say $\varphi$ admits a local
$(\alpha,\lambda)$-extension. As in \cite{cagajipro} local
extension gives full extension.

\begin{theorem}\label{rootns}
Let $j_i:Y_i \iny X_i$ be an isometric embedding for $i=1,...,m$ and
$\alpha$ a finitely generated tensor norm. For a given multilinear
map $\varphi:\prod Y_i \to \mathbb{K}$ with
$\|\varphi\|_{\alpha^*}<\infty$, the following conditions are
equivalent:
\begin{enumerate}
    \item $\varphi$ admits a local
$(\alpha,\lambda)$-extension to $\prod X_i$.
    \item $\varphi$ admits a full extension
to $\prod X_i$, say $\widetilde{\varphi}$, with
$\|\widetilde{\varphi}\|_{\alpha^*}\leq
\lambda\|\varphi\|_{\alpha^*}$.
\end{enumerate}
  \end{theorem}

\begin{proof} For each $i\in \{1,\ldots,m\}$ take an ultrafilter $\mathfrak U_i$
dominating the filter induced by the natural order on the couples
$(E_i,F_i)$ with $E_i\in \fin(Y_i), F_i\in \fin(X_i)$  and
consider the ultrafilter $\mathfrak U=\prod_{i=1}^m \mathfrak U_i$
on $I:=\prod_{i=1}^m \left(\fin(Y_i),\fin(X_i) \right)$. Thanks
to Theorem 3.4 in \cite{FH} we know that the ideal
$\mathcal{A}(Y_1,...,Y_m):=(\widetilde{\otimes}_{\alpha}^m Y_i)^*$
is maximal and ultrastable. For our fixed $\varphi:\prod Y_i \to
\mathbb{K}$ and for each $\iota \in I$,
$\iota=\prod_{i=1}^m(E_i,F_i)$, we have that $\varphi i_E$, where
$i_E:\prod E_i \hookrightarrow \prod Y_i$ is the natural embedding,
admits an extension $\varphi_{\iota}:\prod_{i=1}^m (E_i+F_i)\to
\mathbb{K}$ such that $\varphi_{\iota}\in \mathcal{A}$ and
$\|\varphi_{\iota}\|_{\alpha^*}\leq \lambda \|\varphi
i_E\|_{\alpha^*}\leq \lambda \|\varphi\|_{\alpha^*}$. Obviously
$\sup_{\iota \in I}\|\varphi_{\iota}\|_{\alpha^*}\le \lambda
\|\varphi\|_{\alpha^*}<\infty$ (and consequently $\sup_{\iota \in
I}\|L_{\varphi_{\iota}}\|_{\mathcal{A}}\le \lambda
\|L_{\varphi}\|_{\mathcal{A}}<\infty$).  By the ultrastability of
$\mathcal{A}$, the map
$$\overline{(\varphi_{\iota})}_{\mathfrak U}(\prod_{i=1}^m
(x_{\iota}^i)_{\mathfrak
U_i}):=(\varphi_{\iota}(x_{\iota}^1,...,x_{\iota}^m))_{\mathfrak
U}=\lim_{\mathfrak{U}}\varphi_{\iota}(x_{\iota}^1,...,x_{\iota}^m)$$
defines an $m$-linear map $\prod_{i=1}^n
(E_{i,\iota}+F_{i,\iota})_{\mathfrak U_i}\to \mathbb{K}$ such that
$$\|L_{\overline{(\varphi_{\iota})}_{\mathfrak
U}}\|_{\mathcal{A}}\le\lim_{\mathfrak U}
\|L_{\varphi_{\iota}}\|_{\mathcal{A}}\le \lambda
\|L_{\varphi}\|_{\mathcal{A}}\,,$$ and hence
 $$L_{\overline{(\varphi_{\iota})}_{\mathfrak U}}\in
\mathcal{A}\left((E_{1,\iota}+F_{1,\iota})_{\mathfrak
U_1},...,(E_{n,\iota}+F_{n,\iota})_{\mathfrak
U_n}\right)=\left(\widetilde{\otimes}_{\alpha}^m
(E_{i,\iota}+F_{i,\iota})_{\mathfrak U_i}\right)^* \,.$$ The rest is to
convince us that we have the chain of embedding $\prod Y_i\iny \prod
X_i\iny \prod (E_{i,\iota}+F_{i,\iota})_{\mathfrak U_i} $. To this
end for a point $(y_i)_{i=1}^m\in \prod Y_i $ and
$\iota=\prod(E_i, F_i)$ we define\\

$$y_{i,\iota}:=\left\{
                 \begin{array}{ll}
                   j_i(y_i), & \hbox{if $y_i\in E_i$} \\
                   0, & \hbox{otherwise.}
                 \end{array}
               \right.
$$\\
It follows that $$J:\prod Y_i \iny
\prod(E_{i,\iota}+F_{i,\iota})_{\mathfrak U_i}$$
$$(y_i)\to (y_{1,\iota},...,y_{m,\iota})_{\iota}$$ is an
isometric embedding by the definition of the norm in the
ultraproduct. The same arguments for $\prod X_i$, defining
$$x_{i,\iota}:=\left\{
                 \begin{array}{ll}
                  x_i, & \hbox{if $x_i\in F_i$} \\
                   0, & \hbox{otherwise,}
                 \end{array}
               \right.
$$\\ allow us to get the corresponding isometric embedding $I: \prod X_i \iny
\prod(E_{i,\iota}+F_{i,\iota})_{\mathfrak U_i}$. It is routine to
check that $I\circ(j_i)_{i=1}^m=J$ and thus
$\widetilde{\varphi}:=\overline{(\varphi_{\iota})}_{\mathfrak
U}\circ I$ is the required extension of $\varphi$.
\end{proof}

The notion of locally $\alpha$-extension admits an obviously
symmetric form so it is also possible the ``polynomial version" of
Theorem \ref{rootns}. This ``polynomial version" in terms of
operator ideals works as:

\begin{theorem}\label{rootns1} Let  $\mathcal{Q}(\cdot)$ be an
ideal  of $m$-homogeneous polynomials
and consider an embedding $j:Y \hookrightarrow X$. If a given $q
\in \mathcal{Q}(Y)$ admits a symmetric local extension to $
\mathcal{Q}(X)$ then there exists an extension $\widetilde{q}\in
\mathcal{Q}^{max}(X)$.

\end{theorem}

\begin{proof} The main ingredient of the
proof is the representation theorem of \cite{FH}. Given the operator
ideal $\mathcal{Q}$ we define a tensor norm on finite-dimensional
spaces, as usual, by $\otimes_{\alpha_s}^{m}M:=\mathcal{Q}(M)^*$ and
extend to all normed spaces $Z$ by $$\|z\|_{\otimes_{\alpha_s}^m
Z}:=\inf \{\|z\|_{\otimes_{\alpha_s}^m M} |M\in \fin(Z),z\in
\otimes^{m}_sZ \}.$$  Now, for a given $q\in \mathcal{Q}(Y)$, we
restrict ourselves to the finite-dimensional subspaces of $Y$ and
$X$. This is, $\forall E\in \fin(Y),F\in \fin(X)$ we know that
$q\in
(\otimes_{\alpha_s}^{m}E)^*=\mathcal{Q}(E)^{**}=\mathcal{Q}(E)$
admits an extension to $\widehat{q} \in
(\otimes_{\alpha_s}^{m}(E+F))^*=\mathcal{Q}(E+F)^{**}=\mathcal{Q}(E+F)$
with $\sup \|\widehat{q}\|_{\mathcal{Q}}<\infty$. Similar arguments
as in the Theorem \ref{rootns}, see \cite{FH}, provide us a full
extension $\widetilde{q}\in (\otimes_{\alpha_s}^{m}X)^*$ and using
the representation theorem of \cite{FH} we have that
$(\otimes_{\alpha_s}^{m}X)^*=\mathcal Q^{max}(X)$.
\end{proof}

The preceding result admits also a formulation for multi-linear
operator ideals:

\begin{theorem}
 Let $\mathcal{A}$ be
an ideal of $m$-linear forms and consider embedding $j_i:Y_i
\hookrightarrow X_i$ for $i=1,...,m$. If a given $\varphi \in
\mathcal{A}(Y_1,\ldots,Y_m)$ admits a local extension to $
\mathcal{A}(X_1,\ldots, X_m)$ then there exists a full extension
$\widetilde{\varphi}\in \mathcal{A}^{max}(X_1,\ldots, X_m)$.
\end{theorem}

We define now local $\alpha$-complementation.

 \begin{definition}[Local $\alpha$-complementation]\label{def:alpha-comp}
Let $j:Y\hookrightarrow X$ be an embedding and $\alpha$ a tensor
norm. We say that $Y$ is locally $\alpha$-complemented in $X$
through $j$ if there exists a constant $\lambda$ such that for
every $F\in \fin(X)$, every operator $T:Y\to F^*$  can be extended
to $X$ by an operator $\widetilde{T}:X\to F^*$ with estimate
$\|\widetilde{T}\|_{\mathcal{A}}\leq \lambda \|T\|_{\mathcal{A}}$.
\end{definition}


One has:

\begin{itemize}
\item Every space is locally $\varepsilon$-complemented in any
superspace.

\item Local $\pi$-complementation is a weaker notion than
classical local complementation; i.e., local complementation
implies local $\pi$-complementation. Thus, every Banach space is
locally $\pi$-complemented in its bidual and every
$\mathcal{L}_{\infty}$-space is locally $\pi$-complemented in any
superspace. The converse does not hold: take an uncomplemented
subspace $Y$ of $L_p$, for $2<p<\infty$. Clearly, any $M\in
\fin(L_p)$ verifies that $C_2(M^*)\le T_2(L_p)$, see
\cite{dijato}. By Maurey's extension theorem we can extend any
operator $Y\rightarrow M^*$ to $L_p$ with the norm of the
extension controlled by a universal constant. Hence $Y$ is locally
$\pi$-complemented in $L_p$ but it cannot be locally complemented
because, being $L_p$ is reflexive, it would be complemented.

Moreover, if $X$ contains all finite dimensional Banach spaces,
e.g. $X=C[0,1]$, then $X$ is locally complemented in $E$ if and
only if $X$ is locally $\pi$-complemented in $E$.\\

\item Let $\omega_2$ be the tensor norm associated to the
$\gamma_2$ norm for finite rank operators defined as $$\gamma_2(T)
=\inf \|T_1\|\|T_2\|$$ where the $\inf$ is taken over all possible
factorizations $T = T_1T_2=T$ of $T$ through a Hilbert space. As a
consequence of Maurey's extension theorem one gets  that every
subspace $Y$ of a type 2 Banach space $X$ is locally
$\omega_2$-complemented in $X$ with constant at most $k_G
T_2(X)C_2(X^*)$, see \cite{dijato}.
\end{itemize}

\begin{definition} We say that every $\alpha^*$-integral bilinear form on $Y$, a subspace of $X$,
separately extends to $X$ if the following two conditions hold:
\begin{enumerate}
\item Every $\alpha^*$-integral bilinear form on $Y\times Y$
extends to an $\alpha^*$-integral bilinear form on $Y\times X$.
\item Every $\alpha^*$-integral bilinear form on $Y\times X$
extends to an $\alpha^*$-integral bilinear form on $X\times X$.
\end{enumerate}
\end{definition}

We are ready to connect the extension of bilinear forms with the
notion of local $\alpha$-complementation.

\begin{theorem}\label{rootns2}
Let $\alpha$ be a finitely generated tensor norm. If $Y$ is
locally $\alpha$-complemented in  $X$ then every
$\alpha^*$-integral bilinear form on $Y$ extends separately to
$X$.

\end{theorem}

\begin{proof}
Given a subspace $M\in
\fin(Y)\subset \fin(X)$, we can construct the following
diagram

$$
\begin{CD}
 M\times Y @>>> M\times X \\
  @ViVV  @VViV \\
Y\times Y@>>> Y\times X\\
 @V\varphi VV \\
 \mathbb{K} \\
\end{CD}
$$

By hypothesis $\varphi|_{M\times Y}$ extends to $M\times X$ for
all $M\in \fin(Y)$ and it is therefore enough to apply Theorem
\ref{rootns}. That every bilinear form on $Y\times X$ extends to
$X\times X$ follows the same arguments. Notice that if $Y$ is
locally $\alpha$-complemented in $X$ with constant $\lambda$, then
we get $\lambda^2$ as extension constant for bilinear forms.\\
\end{proof}

\end{document}